\documentclass[12pt,letterpaper,reqno]{amsart}

\usepackage{graphicx}
\usepackage{amsfonts}
\usepackage{amssymb}
\usepackage{amsthm}
\usepackage{amsmath}
\newcommand{\pp}{\ensuremath{\mathbb{P}}}

\newcommand\be{\begin{equation}}
\newcommand\ee{\end{equation}}
\newcommand\nbea{\begin{eqnarray*}}
\newcommand\neea{\end{eqnarray*}}

\newcommand\bea{\begin{eqnarray}}
\newcommand\eea{\end{eqnarray}}
\newcommand\bi{\begin{itemize}}
\newcommand\ei{\end{itemize}}
\newcommand\ben{\begin{enumerate}}
\newcommand\een{\end{enumerate}}
\newcommand\bc{\begin{center}}
\newcommand\ec{\end{center}}
\newcommand\ba{\begin{array}}
\newcommand\ea{\end{array}}

\newcommand{\E}{\ensuremath{\mathbb{E}}}

\newcommand{\ce}{\mathcal{E}}

\newcommand{\foh}{\frac{1}{2}}  

\newtheorem{thm}{Theorem}[section]

\newtheorem{lem}[thm]{Lemma}

\newtheorem{defi}[thm]{Definition}

\newcommand{\ncr}[2]{{#1 \choose #2}}

\newcommand{\gep}{\epsilon}

\numberwithin{equation}{section}

\begin{document}

\title{On the number of summands in Zeckendorf decompositions}

\author{Murat Kolo$\breve{{\rm g}}$lu}\email{Murat.Kologlu@williams.edu}
\address{Department of Mathematics and Statistics, Williams College,
Williamstown, MA 01267}

\author{Gene Kopp}\email{gkopp@uchicago.edu}
\address{Department of Mathematics, University of Chicago, Chicago, IL 60637}

\author{Steven J. Miller}\email{Steven.J.Miller@williams.edu}
\address{Department of Mathematics and Statistics, Williams College,
Williamstown, MA 01267}

\author{Yinghui Wang}\email{yinghui@mit.edu}
\address{Department of Mathematics, MIT, Cambridge, MA 02139}



\date{\today}

\thanks{The first, second and fourth named authors were partially supported by NSF Grants DMS0855257, DMS0970067, Williams College and the MIT Mathematics Department, and the second named author was partially supported by NSF grant DMS0850577. It is a pleasure to thank our colleagues from the Williams College 2010 SMALL REU program for many helpful conversations (especially the Diophantine Arithmetic and Analysis group of Ed Burger, David Clyde, Cory Colbert, Gea Shin and Nancy Wang), and Ed Scheinerman for useful comments on an earlier draft.}


\begin{abstract}
Zeckendorf proved that every positive integer has a unique representation as a sum of non-consecutive Fibonacci numbers. Once this has been shown, it's natural to ask how many summands are needed. Using a continued fraction approach, Lekkerkerker proved that the average number of such summands needed for integers in $[F_n, F_{n+1})$ is $n / (\varphi^2 + 1) + O(1)$,
where $\varphi = \frac{1+\sqrt{5}}2$ is the golden mean. Surprisingly, no one appears to have investigated the distribution of the number of summands; our main result is that this converges to a Gaussian as $n\to\infty$. Moreover, such a result holds not just for the Fibonacci numbers but many other problems, such as linear recurrence relation with non-negative integer coefficients (which is a generalization of base $B$ expansions of numbers) and far-difference representations.

In general the proofs involve adopting a combinatorial viewpoint and analyzing the resulting generating functions through partial fraction expansions and differentiating identities. The resulting arguments become quite technical; the purpose of this paper is to concentrate on the special and most interesting case of the Fibonacci numbers, where the obstructions vanish and the proofs follow from some combinatorics and Stirling's formula; see \cite{MW} for proofs in the general case. \end{abstract}

\maketitle

\tableofcontents



\section{Introduction}

The Fibonacci numbers are one of the most well known and studied sequences in mathematics, as well as one of the most enjoyable to play with. There are books (such as \cite{Kos}) and journals (such as the Fibonacci Quarterly) dedicated to all their wondrous properties. The purpose of this article is to review two nice results, namely Zeckendorf's and Lekkerkerker's Theorems, and discuss some massive generalizations.

Before stating our results, we first set some notation. We label the Fibonacci numbers by $F_1 = 1, F_2 = 2, F_3
= 3, F_4 = 5$ and in general $F_n = F_{n-1}+F_{n-2}$; we'll discuss shortly why it is convenient to use this non-standard counting. We let $\varphi$ denote the golden mean, $\frac{1+\sqrt{5}}2$, which satisfies $\varphi^2 = \varphi + 1$.

Zeckendorf (see for instance \cite{Ze}) proved that every positive integer can be written uniquely as a sum of non-adjacent Fibonacci numbers. The proof is a straightforward induction. Note, though, how important it is that our series begins with just a single 1; if we had two 1s then the decompositions of many numbers into non-adjacent summands would not be unique.

In the 1950s, Lekkerkerker \cite{Lek} answered the following question, which is a natural outgrowth of Zeckendorf's theorem: \emph{On average, how many summands are needed in the Zeckendorf decomposition?} Lekkerkerker proved that for integers in $[F_n, F_{n+1})$ the average number of summands, as $n\to\infty$, is $\frac{n}{\varphi^2+1}+O(1)$ $\approx$ $.276n$.

Of course, one can ask these questions for more general recurrence relations. Zeckendorf's result has been generalized to several recurrence relations (see the 1972 special volume on representations in the Fibonacci Quarterly, especially \cite{Ho,Ke}, as well as \cite{Len}). Burger \cite{Bu} proved the analogous result for the mean number of summands for a generalization of Fibonacci numbers, $G_{n} = G_{n-1} + G_{n-2} + \cdots + G_{n-L}$. There is, of course, another generalization, which interestingly does not seem to have been asked. Namely, \emph{how are the number of summands distributed about the mean for integers in $[F_n, F_{n+1})$.} This is a very natural question to ask. Both the question and the answer are reminiscent of the Erd\H{o}s-Kac Theorem \cite{EK}, which states that as $n\to\infty$ the number of distinct prime divisors of integers on the order of size $n$ tends to a Gaussian with mean $\log\log n$ and standard deviation $\sqrt{\log\log n}$.

Our main result is that a similar statement about Gaussian behavior holds, not just for the Fibonacci numbers, but for the large class of recurrence relations defined below.

\begin{defi}\label{defn:goodrecurrencereldef} We say a sequence $\{H_n\}_{n=1}^\infty$ of positive integers is a \textbf{Positive Linear Recurrence Sequence (PLRS)} if the following properties hold:

\ben

\item \emph{Recurrence relation:} There are non-negative integers $L, c_1, \dots, c_L$ such that $$H_{n+1} \ = \ c_1 H_n + \cdots + c_L H_{n+1-L},$$ with $L, c_1$ and $c_L$ positive.

\item \emph{Initial conditions:} $H_1 = 1$, and for $1 \le n < L$ we have
$$H_{n+1} \ =\
c_1 H_n + c_2 H_{n-1} + \cdots + c_n H_{1}+1.$$

\een

We call a decomposition $\sum_{i=1}^{m} {a_i H_{m+1-i}}$ of a positive integer $N$ (and the sequence $\{a_i\}_{i=1}^{m}$) \textbf{legal} if $a_1>0$, the other $a_i \ge 0$, and one of the following two conditions holds:

\bi

\item We have $m<L$ and $a_i=c_i$ for $1\le i\le m$.

\item There exists $s\in\{0,\dots, L\}$ such that
\begin{equation}\label{scon}
a_1\ =\ c_1,\ \ \ a_2\ =\ c_2,\ \ \ \cdots,\ \ \ a_{s-1}\ = \ c_{s-1}\ \ \ {\rm{and}}\  \ \ a_s\ < \ c_s,
\end{equation} $a_{s+1}, \dots, a_{s+\ell} = 0$ for some $\ell \ge 0$,
and $\{b_i\}_{i=1}^{m-s-\ell}$ (with $b_i = a_{s+\ell+i}$) is legal.
\ei

If $\sum_{i=1}^{m} {a_i H_{m+1-i}}$ is a legal decomposition of $N$, we define the \textbf{number of summands} (of this decomposition of $N$) to be $a_1 + \cdots + a_m$.
\end{defi}

Informally, a legal decomposition is one where we cannot use the recurrence relation to replace a linear combination of summands with another summand, and the coefficient of each summand is appropriately bounded. For example, if $H_{n+1} = 2 H_n + 3 H_{n-1} + H_{n-2}$, then $H_5 + 2 H_4 + 3 H_3 + H_1$ is legal, while $H_5 + 2 H_4 + 3 H_3 + H_2$ is not (we can replace $2 H_4 + 3 H_3 + H_2$ with $H_5$), nor is $7H_5 + 2H_2$ (as the coefficient of $H_5$ is too large). Note the Fibonacci numbers are just the special case of $L=2$ and $c_1 = c_2 = 1$.

We adopt a probabilistic language to state our main results.

\begin{defi}[Associated Probability Space to a Positive Linear Recurrence Sequence]\label{def:assocprobspace}
Let $\{H_n\}$ be a Positive Linear Recurrence Sequence. For each $n$, consider the discrete outcome space \be \Omega_n \ = \ \{H_n,\ H_n+1,\ H_n + 2,\  \cdots,\ H_{n+1}-1\} \ee with probability measure \be \pp_n(A) \ = \ \sum_{\omega \in A \atop \omega \in \Omega_n} \frac1{H_{n+1}-H_n}, \ \ \  a \subset \Omega_n; \ee in other words, each of the $H_{n+1}-H_n$ numbers is weighted equally. We define the random variable $K_n$ by setting $K_n(\omega)$ equal to the number of summands of $\omega \in \Omega_n$ in its legal decomposition. Implicit in this definition is that each integer has a unique legal decomposition; we will prove this fact, and thus $K_n$ is well-defined. It is also convenient to study $\mathcal{K}_n = K_n-1$; every integer in $\Omega_n$ must have at least on $H_n$ as a summand; thus $\mathcal{K}_n$ is the number of non-forced summands in a legal decomposition.
\end{defi}

Our main result is

\begin{thm}\label{thm:main} If $\{H_n\}_{n=1}^\infty$ is a \emph{Positive Linear Recurrence Sequence} then every positive integer may be written uniquely as a legal sum $\sum_i a_i H_i$ (see Definition \ref{defn:goodrecurrencereldef}). Let $K_n$ be the random variable of Definition \ref{def:assocprobspace}. Then $\E[K_n]$ and ${\rm Var}(K_n)$ are of order $n$, and as $n\to\infty$, $K_n$ (and $\mathcal{K}_n)$ converges to a Gaussian.
\end{thm}

Our result first extends Zeckendorf's and Lekkerkerker's theorems to a large class of recurrence relations, and then goes further and yields the distribution of summands tends to a Gaussian. The proof has three main ingredients. The first is to adopt a combinatorial point of view. Previous approaches to Lekkerkerker's theorem were number theoretic, involving continued fractions. We instead view this as a combinatorial problem, namely how many ways can we choose elements in a set subject to some restrictions on what may be taken. This approach gives us an explicit formula for the number of $N$ in $[H_n, H_{n+1})$ that have a given number of summands. For each interval $[H_n, H_{n+1})$ (i.e., for each $n$) we may thus associate a probability function $p_n$ on $[H_n, H_{n+1})$ where $p_n(k)$ is the probability of having exactly $k+1$ summands;\footnote{As remarked earlier, we choose to write the density this way as every integer in $[H_n, H_{n+1})$ must have $H_n$ in its decomposition, and thus it is more natural to study the number of additional, non-forced summands needed.} note $p_n$ is the density of the random variable $\mathcal{K}_n$. We then use generating functions and differentiating identities to obtain tractable formulas for these summand functions $p_n$, and prove our claims about the mean and the variance. We conclude by showing that as $n\to\infty$ the centered and normalized moments of $p_n$ tend to the moments of the standard normal; by Markov's Method of Moments this yields the Gaussian behavior.\\

We can gain a lot of intuition as to why these results are true by looking at the special case of $L=1$ and $c_1 = B > 0$. For these choices, our sequence is just $H_n = B^{n-1}$; in other words, we are looking at the base $B$ decomposition of integers. Zeckendorf's theorem is now clearly true, as every number has a unique representation. Lekkerkerker and the Gaussian behavior are now just consequences of the Central Limit Theorem. For example, consider a decomposition of an $N \in [B^n, B^{n+1})$: $$N \ = \ a_1 B^n + a_2 B^{n-1} + \cdots + a_{n+1} B^0.$$ We have $a_1 \in \{1,\dots, B-1\}$ and all other $a_i \in \{0,\dots,B-1\}$. We are interested in the behavior, for large $n$, of $a_1 + \cdots + a_{n+1}$ as we vary over $N$ in $[B^n, B^{n+1})$. Note for large $n$ the contribution of $a_1$ is immaterial, and the remaining $a_i$'s can be understood by considering the sum of $n$ independent, identically distributed uniform random variables on $\{0, \dots, B-1\}$ (which have mean $\frac{B-1}2$ and standard deviation $\sqrt{(B^2-1)/12}$). Denoting these by $A_i$, by the Central Limit Theorem $A_2 + \cdots + A_{n+1}$ converges to being normally distributed with mean $\frac{B-1}2 n$ and standard deviation $n\sqrt{(B^2-1)/12}$. \\

Our approach is quite general, and can handle a variety of related problems. We state just one more, which allows us to see some very interesting behavior.

Recently Alpert \cite{Al} showed that every positive integer can be written uniquely as a sum and difference of the Fibonacci numbers  $\{F_n\}_{n=1}^\infty$ such that every two terms of the same sign differ in index by at least 4, and every two terms of opposite sign differ in index by at least 3; we call this the \textbf{far-difference representation}. For example, $$2011 \ = \ F_{17} - F_{14} + F_8 + F_3 \ \ \ {\rm and} \ \ \ 1900 \ = \ F_{17} - F_{14} - F_{10} + F_6 + F_2.$$ If $S_n=\sum_{0<n-4i\le n} F_{n-4i} = F_n+F_{n-4}+\cdots$  for positive $n$ and 0 otherwise, then for each $N\in (S_{n-1},S_n]$ the first term in its far-difference representation is $F_n$. Note that 0 has the empty representation. We can show

\begin{thm}\label{thm:hannah} Consider the outcome space  $\Omega_n = \{S_{n-1}+1, S_{n-1}+2, \ \dots, \ S_n\}$ with probability measure $\pp_(A) = \sum_{\omega \in A} \frac1{S_n-S_{n-1}}$ for $A \subset \Omega_n$. Let $K_n$ and $L_n$ be the random variables denoting the number of positive and negative Fibonacci summands in the far-difference representation (they are well-defined by \cite{Al}). As $n\to\infty$, for any real numbers $a$ and $b$ the random variable $aK_n+bL_n$ converges to a Gaussian. The expected value of $K_n$ (which is $\phi/2$ more than that of $L_n$) is $\frac{n}{10} + \frac{371-113\sqrt{5}}{40} + o(1)$; the variance of both is of size $\frac{15+21\sqrt{5}}{1000}n$. $K_n$ and $L_n$ are negatively correlated, with a correlation coefficient of $-(21-2\varphi)/(29+2\varphi)$. Further, $K_n+L_n$ and $K_n-L_n$ are independent random variables as $n\to\infty$, which implies the total number of Fibonacci numbers is independent of the excess of positive to negative summands.
\end{thm}

Unfortunately, our arguments become involved and technical to handle all of the cases in Theorems \ref{thm:main} and \ref{thm:hannah}. In order to highlight the ideas without getting bogged down in computations, in this note we concentrate on the most important special case of Theorem \ref{thm:main}, namely the Fibonacci numbers, and provided a sketch of some of the arguments in Theorem \ref{thm:hannah}. We first describe our combinatorial perspective, which yields \be\label{eq:formulapnkzeck}p_n(k)\ =\ \frac{\ncr{n-1-k}{k}}{F_{n-1}} \ee (remember $p_n(k)$ is the probability that an $N \in [F_n, F_{n+1})$ has exactly $k+1$ summands in its Zeckendorf decomposition). All of our theorems follow from knowing this density. The difficulties in the general case are due to the fact that the corresponding formulas for the densities are far more involved; here we can easily determine the behavior by applying Stirling's formula.

We first use our explicit formula for $p_n(k)$ to prove Zeckendorf's theorem, and then sketch how we can use it to prove Lekkerkerker's theorem (the mean $\mu_n$) as well as compute the variance $\sigma_n^2$. We then show that as $n\to\infty$, $p_n(k)$ converges to a Gaussian with mean $\mu_n$ and variance $\sigma_n^2$. While technically we could just immediately jump to the limiting behavior of the density, it would be very unmotivated not knowing the mean and the variance and either choosing the correct values by divine inspiration, or having them fall out of the resulting algebra.


\section{Combinatorial Perspective and Zeckendorf's Theorem}

The key input in our analysis is counting the number of solutions to a well known Diophantine equation.\footnote{This problem is also known as the stars and bars problem, the cookie problem, or the simplest case of Waring's problem.}

\begin{lem}\label{lem:cookieproblem}
\begin{enumerate}\ \\
\item The number of
ways of dividing $n$ identical objects among $p$ distinct people is
$\ncr{n+p-1}{p-1}$. Equivalently, this is the number of solutions to
$y_1 + \cdots + y_p = n$ with each $y_i$ a non-negative integer.
\item More generally, the number of solutions to $y_1 + \cdots + y_p = n$ with $y_i \ge c_i$ (each $c_i$ a non-negative integer) is $\ncr{n-(c_1+\cdots+c_p) + p-1}{p-1}$.
\end{enumerate}
\end{lem}

\begin{proof} For (1): The two formulations are clearly equivalent; simply
interpret $y_i$ as the number of objects person $i$ receives. To
prove the claimed formula, imagine the $m$ objects are in a row and
we add $p-1$ items at the end. We now have $n+p-1$ objects. There is a one-to-one correspondence between assigning the $n$ objects to the $p$ people and choosing $p-1$ of $n+p-1$ objects.
There are $\ncr{n+p-1}{p-1}$ ways to choose $p-1$ of the $n+p-1$ items. All the
items up to the first one chosen go to person 1, then the items
up to the second one chosen go to person 2, and so on.

For (2): We may write $y_i = x_i+c_i$ with each $x_i \ge 0$ a non-negative integer. Our problem is equivalent to $$(x_1+c_1) + \cdots + (x_p + c_p) \ = \ n, $$ which becomes $$x_1 + \cdots + x_p \ = \ n - (c_1 + \cdots + c_p), \ \ \ x_i \ge 0,$$ whose solution is given by part (1).  \end{proof}

There are two parts to Zeckendorf's Theorem: not only does a
decomposition exist of any positive integer as a sum of
non-consecutive Fibonacci numbers, but such a decomposition is
unique. As our combinatorial approach does require this uniqueness as an
input, we provide the standard proof below.

\begin{lem}[Zeckendorf's Theorem - Uniqueness of
Decomposition]\label{lem:zeckunique} If two sums of non-consecutive
Fibonacci numbers are equal, then the two sums have the same summands.
\end{lem}

\begin{proof} Assume \be\label{eq:twodecompN}  F_{n_1} + F_{n_2} + \cdots + F_{n_k} \ = \ F_{m_1} + F_{m_2} +
\cdots + F_{m_\ell}\ee where $n_1 \ge n_2 \ge \cdots$ and $m_1 \ge m_2 \ge \cdots$. Without loss of generality we may assume $F_{n_1} > F_{m_1}$ (as otherwise we would just remove some summands). As each decomposition is of non-adjacent summands, if we add 1 to the decomposition on the right of \eqref{eq:twodecompN}, the largest it can be\footnote{For example, $F_6+F_4+F_2+1$ $=$ $F_6+F_4+F_3$ $=$ $F_6+F_5$ $=$ $F_7$.} is $F_{m_1} + F_{m_2+1}$, which itself is at most $F_{m_1+1}$. As $F_{n_1} > F_{m_1}$, we see that adding 1 to the right hand side of \eqref{eq:twodecompN} yields a number at most $F_{n_1}$; thus \be F_{n_1} + F_{n_2} + \cdots + F_{n_k} \ > \ F_{m_1} + F_{m_2} +
\cdots + F_{m_\ell}, \ee contradiction. \end{proof}


We are now in a position to prove the first part of Zeckendorf's Theorem. In the course of our proof we will derive \eqref{eq:formulapnkzeck}, the claimed formula for $p_n(k)$.

\begin{thm}[Zeckendorf's Theorem - Existence of
Decomposition]\label{thm:zecknonconsec} Any natural number can be
expressed as a sum of non-consecutive Fibonacci numbers. Further, the probability an $N \in [F_n, F_{n+1})$ has exactly $k+1$ summands in its Zeckendorf decomposition is $\ncr{n-1-k}{k}/F_{n-1}$ (in other words, the density function of the random variable $\mathcal{K}_n$ is $p_n(k) = \ncr{n-1-k}{k}/F_{n-1}$.
\end{thm}

\begin{proof} Consider all $N$ in $[F_n, F_{n+1})$. The
number of integers in this interval is $F_{n+1}-F_n=F_{n-1}$. We
claim that each of these $F_{n-1}$ integers can be expressed as a
sum of a certain subset of $\{F_1, F_2,\ldots, F_n\}$ with the
properties that no two consecutive Fibonacci numbers appear in the
sum and that $F_n$ is one of the summands.  From the arguments in the proof of Lemma \ref{lem:zeckunique}, it is clear that $F_n$ must appear in this sum since otherwise the sum would be too small to
be an element of $[F_n, F_{n+1})$.

We now translate the previous claim into the combinatorial formulation of Lemma \ref{lem:cookieproblem}. Suppose we wish to have $k$ summands in
addition to the summand $F_n$ in our sum, with no two summands adjacent. Clearly $k \le \lfloor
\frac{n-1}2\rfloor$. Choosing a valid
set of $k$ summands, say $F_{m_1}, \dots, F_{m_k}$,
is equivalent to choosing $k$ indices $m_1, m_2, \ldots, m_k$ from
the set $\{1,2,\ldots,n-1\}$, with the property that $m_i < m_{1+1}
-1$ and $n-1$ is not chosen (as otherwise we would have the adjacent summands $F_{n-1}$ and $F_n$).

We may assume $k > 0$, as there is only one way to choose no additional summands.
We define the auxiliary sequence $y_j$ as follows: $y_0 = m_1
-1$, and for $1\leq j \leq k$, $y_j = m_{j+1}-m_j -1$ (as noted
earlier, $m_{k+1} = n$). For example, with $n=9$, $k=3$, and the
sequence $F_1 + F_3 + F_6 + F_9$, we have $y_0 = 0$, $y_1 = 1$, $y_2
= 2$ and $y_3 = 2$. Note $y_0$ is the number of indices before the
first index chosen, and for $j>0$, $y_j$ equals the number of unused indices
between $m_j$ and $m_{j+1}$.  Clearly we have $y_0 \ge 0$ and $y_j
\geq 1$ for $j>0$.  Thus we have used $k+1$ indices (including the
required index $n$ associated with the largest summand, $F_n$) and
$y_0+y_1+\cdots +y_k$ unused indices from $\{1,2,\ldots,n\}$.  Hence
we have
\be\label{eq:yeqn} (k+1) + y_0 + y_1 +\cdots +y_k\ =\ n.\ee

If we make the change of of variables $x_0= y_0$ and $x_j = y_j-1$
for $j>0$, then we have $x_j\geq 0$ for all $j$
with no other constraints on these variables other than they must
satisfy the identity implicitly given by \eqref{eq:yeqn}; that is
\be\label{eq:keyeqtocookiezeck} x_0 + \cdots + x_k \ = \ n-1-2k.\ee
In other words, in view of Lemma \ref{lem:zeckunique} we have a
bijection between the set of Zeckendorf decompositions with $k+1$
summands having $F_n$ as its largest summand and the set of all
non-negative integer solutions to  \eqref{eq:keyeqtocookiezeck}.
By Lemma \ref{lem:cookieproblem}, the number of
solutions to \eqref{eq:keyeqtocookiezeck} is
$\ncr{n-1-2k+(k+1-1)}{k+1-1}$ $=$ $\ncr{n-1-k}{k}$. Thus the number
of Zeckendorf decompositions having largest summand $F_n$ is
precisely \be \sum_{k=0}^{\lfloor \frac{n-1}2\rfloor}
\ncr{n-1-k}{k},\ee which, by a well-known identity for binomial sums,
equals $F_{n-1}$ (see Lemma \ref{lem:sumupdownfib} of Appendix
\ref{sec:appcombid} for a proof). As remarked, by
Lemma \ref{lem:zeckunique} each one of these sequences gives rise to
a distinct Zeckendorf sum. Thus the number of Zeckendorf
decompositions in the interval $[F_n, F_{n+1})$ is equal to
$F_{n-1}$, which is the total number of integers in that
interval.  As $n$ was arbitrary, and these intervals partition the
set of natural numbers, we have shown that every natural number has a
unique Zeckendorf decomposition, and the number of $N$ in $[F_n, F_{n+1})$ with exactly $k+1$ summands in its decomposition is $\ncr{n-1-k}{k}$. In other words, the probability of a number in this interval having precisely $k+1$ summands is $\ncr{n-1-k}{k}/F_{n-1}$.
\end{proof}


\section{Lekkerkerker's Theorem and the Variance}

We sketch how our approach easily yields Lekkerkerker's theorem. We only provide a sketch as, of course, Lekkerkerker's theorem follows immediately from our proof of the Gaussian behavior. We highlight the key steps as Lekkerkerker's theorem is of interest in its own right, and it is nice to have a new, elementary proof of it (our proof of the Gaussian behavior will involve Stirling's formula, and is different than the arguments below).

The average number of summands needed in the Zeckendorf decomposition is just \bea\label{eq:avenumbsummands} \mu_n & \ = \ & \sum_{k=0}^{\lfloor \frac{n-1}2\rfloor} (k+1) \frac{\ncr{n-1-k}{k}}{F_{n-1}} \nonumber\\ &=& 1 + \frac1{F_{n-1}} \sum_{k=0}^{\lfloor \frac{n-1}2\rfloor} k \ncr{n-1-k}{k} \nonumber\\ &=& 1 + \frac{\mathcal{E}(n)}{F_{n-1}}. \eea Thus the problem is reduced to computing \be \mathcal{E}(n) \ = \ \sum_{k=0}^{\lfloor \frac{n-1}2\rfloor} k \ncr{n-1-k}{k}. \ee

We can determine a closed-form expression for $\ce(n)$ by first showing that
it satisfies a certain recurrence relation.

\begin{lem}[Recurrence relation for $\ce(n)$]\label{lem:recurrencecen}  We have \be
\ce(n)+\ce(n-2) \ = \ (n-2)F_{n-3}.\ee \end{lem}

The proof follows from straightforward algebra, and is given in Appendix \ref{sec:appcombid}. Solving the recurrence relation yields

\begin{lem}[Formula for $\ce(n)$]\label{lem:formula}  We have \be
\ce(n) \ = \ \frac{n F_{n-1}}{\varphi^2+1} + O(F_{n-2}).\ee
\end{lem}

The proof follows from using telescoping sums to get an
expression for $\ce(n)$, which is then evaluated by inputting Binet's
formula\footnote{There are many ways of proving Binet's formula. One of the simplest is through generating functions and partial fractions; a generalization of this plays a key role of the proof of the general case of Theorem \ref{thm:main} (see \cite{MW}).} and differentiating identities. Recall Binet's formula (with our notation) asserts \be\label{eq:binetformula} F_n  \ = \ \frac{\varphi}{\sqrt{5}} \cdot  \varphi^n - \frac{1-\varphi}{\sqrt{n}}\cdot (1-\varphi)^n, \ \ \ \varphi \ = \ \frac{1+\sqrt{5}}2.\ee The details are provided in Appendix \ref{sec:appcombid}.

Lekkerkerker's Theorem now follows immediately by substituting the result for $\ce(n)$ from Lemma \ref{lem:formula} into the equation for the mean, \eqref{eq:avenumbsummands}. Explicitly,

\begin{thm}\label{thm:number} The average number of non-consecutive
Fibonacci summands used in representing numbers in $[F_n, F_{n+1})$
is \be\label{eq:meann} \E[K_n] \ = \ \mu_n \ = \ \frac{5-\sqrt{5}}{10}\ n \ - \ \frac{2}{5}\ =\ \frac{1}{\varphi^2+1}\ n\ - \ \frac{2}{5} \ = \ \mathbb{E}[\mathcal{K}_n]+1.\ee
\end{thm}


A similar calculation shows

\begin{thm} The variance in the number of non-consecutive Fibonacci used in representing numbers in $[F_n, F_{n+1})$
is \be {\rm Var}(K_n) \ = \ \sigma_n^2 \ = \ \frac{1}{5\sqrt{5}}\ n\
 - \ \frac{2}{25}\ = \ \frac{\varphi}{5(\varphi+2)}\ n\ - \ \frac{2}{25} \ = \ {\rm Var}(\mathcal{K}_n).\ee
\end{thm}


\section{Gaussian Behavior}\label{sec:erdoskactype}

In the last section we computed the mean by evaluating certain combinatorial sums. We could similarly derive an explicit formula for the variance, or more generally, any moment. We choose instead to analyze the density $p_n(k)$ in greater detail, and show that it converges pointwise to a Gaussian with mean $\mu_n$ approximately $\frac{n}{\varphi^2+1}=\frac{n}{\varphi+2}$ and variance $\sigma_n^2 \approx \frac{\varphi n}{5(\varphi+2)}$. In the course of proving this convergence, the mean and the standard deviation naturally fall out of the calculation. While this does make the previous section superfluous, we chose to include it as it provides an elementary proof of Lekkerkerker (as well as telling us what the mean and variance are, which are a great aid in performing the Stirling analysis below\footnote{If we didn't know $\mu_n$ and $\sigma_n$ we would just keep these as initially free parameters, and then choose the values appropriately to ensure the limits below exist.}).

Before delving into the proof of Theorem \ref{thm:main}, we provide some evidence by looking at the number of summands in the
Zeckendorf decomposition for integers in $[F_{2010}, F_{2011})$ (see
Figure \ref{fig:GaussianBehaviorNumFibands}); the fit is visually striking.
\begin{figure}
\begin{center}
\includegraphics[width=4.0in]{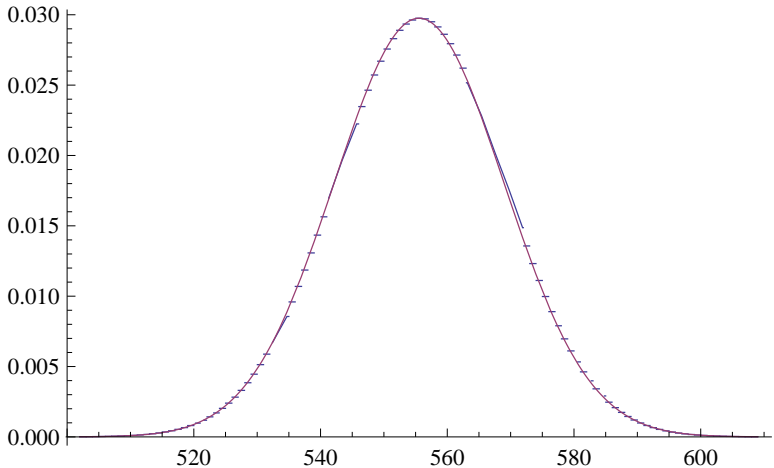}
\caption{\label{fig:GaussianBehaviorNumFibands} Plot of the probability $p_{2010}(k)$ of having exactly $k+1$ summands in the Zeckendorf decompositions of integers in $[F_{2010}, F_{2011})$ (with $p_n(k) = \ncr{n-1-k}{k}/F_{n-1}$) against the Gaussian with mean $\mu_{2010} \approx 555.55$ and variance $\sigma_{2010}^2 \approx 179.78$.}
\end{center}
\end{figure}

\subsection{Preliminaries}

The computation is a little cleaner if instead of studying the interval $[F_n, F_{n+1})$ we instead study $[F_{n+1},F_{n+2})$. The density is now \be p_{n+1}(k) \ = \ \frac{\ncr{n-k}{k}}{F_n}. \ee  We list some useful expansions: \bea \mu_{n+1} & \ = \ & \frac{n}{\varphi+2} + O(1) \nonumber\\ \sigma_{n+1}^2 & = & \frac{\varphi n}{5(\varphi + 2)} + O(1) \nonumber\\ F_n & = & \frac{\varphi}{\sqrt{5}} \varphi^n + O\left((1-\varphi)^n\right). \eea

We will expand $\ncr{n-k}{k}$ using Stirling's formula, which says for large $m$ that \bea m! & \ = \ & m^m e^{-m} \sqrt{2\pi m} \left(1 + \frac1{12m} + \frac1{288m^2} - \frac{139}{51840m^3} + \cdots\right)\nonumber\\ & \ = \ & m^m e^{-m} \sqrt{2\pi m}\left(1 + O\left(\frac1m\right)\right).\eea As the mean and variance\footnote{While we have not proved the variance is of size $n$, a similar calculation as that for the mean yields it without any trouble; in the interest of space we merely state the result in Theorem \ref{thm:number}. We could also get the right order of magnitude for the variance by the Stirling computation that follows.} are both of size $n$, for any fixed $\gep > 0$ as $n\to\infty$ there is negligible probability of all $k$ with $|k-\mu_{n+1}| > n^{1/2+\gep}$.\footnote{This follows by Chebyshev's inequality. As we are on the order of $n^\epsilon$ standard deviations from the mean, the probability is at most $O(1/n^{2\epsilon})$, which tends to zero with $n$.} Thus we need only worry about analyzing $p_{n+1}(k)$ for $|k-\mu_{n+1}| \le n^{1/2+\gep}$. For such $k$, we have $k$, $n-k$ and $n-2k$ will all be of size $n$ and hence large, and thus Stirling's formula yields a great approximation.

After some simple algebra, which includes using Binet's formula (see \eqref{eq:binetformula}) for $F_n$, we find
\nbea p_{n+1} (k) & = & \ncr {n-k}{k} \frac{1}{F_n}
\\ & = & \frac{(n-k)!}{k!(n-2k)!}\ \frac{1}{F_n}
\\ & = & \frac1{\sqrt{2\pi}} \frac{(n-k)^{n-k+\foh} \left(1+O(\frac1{n-k})\right)}{k^{k+\foh} \left(1+O\left(\frac1{k}\right)\right) (n-2k)^{n-2k + \foh}  \left(1 + O(\frac1{n-2k})\right)} \ \frac{\sqrt{5}}{\varphi \cdot \varphi^n \left(1+O\left(\frac1{\varphi^n}\right)\right)} \\ & = & \frac1{\sqrt{2\pi}} \frac{(n-k)^{n-k+\foh}}{k^{k+\foh} (n-2k)^{n-2k + \foh}} \ \frac{\sqrt{5}}{\varphi \cdot \varphi^n} \left(1+O\left(\frac1{n}\right)\right)\\  &=& f_{n+1}(k) \left(1 + O\left(\frac1n\right)\right).
\neea

It suffices to analyze $f_{n+1}(k)$ for $|k-\mu_{n+1}| \le n^{1/2+\gep}$, as the $O(1/n)$ term leads to negligible changes in the density function. We now split off the terms that exponentially depend on $n$, and write
\nbea f_{n+1} (k) & = & \left(\frac{1}{\sqrt{2 \pi}} \sqrt{\frac{(n-k)}{k(n-2k)}} \frac{\sqrt{5}}{\varphi} \right) \left( \varphi^{-n} \frac{(n-k)^{n-k}}{k^{k} (n-2k)^{n-2k}} \right) \nonumber\\ &=& N_n(k) S_n(k).
\neea

We change variables and replace $k$ with its distance from the mean in units of the standard deviation, $\sigma_{n+1}$. Thus we write \be k \ = \ \mu_{n+1} + x \sigma_{n+1}, \ee and \be f_{n+1}(k) dk \ = \  f_{n+1}(\mu_{n+1} + \sigma_{n+1} x)\sigma_{n+1} dx. \ee It is essential that we record $dk$ transforms to $\sigma_{n+1} dx$. In these new variables the main action occurs at $x=0$, and the scale is on the order of $1$; in other words, once $x$ is large (such as $n^{\epsilon}$) then we are many standard deviations away and the density is negligible.

The next few pages are the detailed computation. While the computations are long in places, the basic idea (the combinatorial perspective) is straightforward: viewing the problem combinatorially yields an explicit density function, whose large $n$ asymptotics follow from Stirling's formula.

\subsection{Analysis of $N_n(k)$}

\begin{lem}\label{lem:thmmainnnk} For any $\gep > 0$ we have \be N_n(k)\ =\ \frac{1}{\sqrt{2 \pi \sigma_{n+1}^2}} \cdot\left(1 + O\left(n^{-1/2+\gep}\right)\right). \ee \end{lem}

Let $C = \frac{1}{\varphi+2}$, so $\mu_{n+1} = Cn + O(1)$,  and recall $\sigma_{n+1}^2 = \frac{\varphi n}{5(\varphi+2)} + O(1)$. We use a change of variable and some algebra to simplify $N_n(k)$. We set \be u  \ = \ \frac{\sigma_{n+1}}{n}x \ \approx \ \frac{x}{\sqrt{n}} \ \ll \ O(n^{-1/2+\gep}), \ee   and find
\nbea
N_n(k) &\ =\ & \frac{1}{\sqrt{2 \pi}} \sqrt{\frac{n-k}{k(n-2k)}} \frac{\sqrt{5}}{\varphi}
\\ & = & \frac{1}{\sqrt{2 \pi n}} \sqrt{\frac{1-k/n}{(k/n)(1-2k/n)}} \frac{\sqrt{5}}{\varphi}
\\ & = & \frac{1}{\sqrt{2 \pi n }} \sqrt{\frac{1- (\mu_{n+1}+\sigma_{n+1} x)/n}{((\mu_{n+1}+\sigma_{n+1} x)/n)(1-2(\mu_{n+1}+\sigma_{n+1} x)/n)}} \frac{\sqrt{5}}{\varphi}
\\ & = & \frac{1}{\sqrt{2 \pi n }} \sqrt{\frac{1- C - u}{(C+u)(1-2C-2u)}} \frac{\sqrt{5}}{\varphi} \cdot \left(1 + \frac1n\right),
\neea where the last error arises from replacing $\mu_{n+1}$ with $Cn + O(1)$. In fact, as $u = O(n^{-1/2+\gep})$ we may drop the $u$'s at the cost of replacing the error term $O(n^{-1})$ with $O(n^{-1/2+\gep})$. The following relations help simplify our expression for $N_n(k)$: $$\varphi+1 \ = \ \varphi^2, \ \ \ C = \frac1{\varphi+2}, \ \ \ 1-C \ = \ \frac{\varphi+1}{\varphi+2}, \ \ \ 1 - 2C \ = \ \frac{\varphi}{\varphi+2}.$$
Using these, as well as $\sigma_{n+1}^2 = \frac{\varphi n}{5(\varphi+2)} + O(1)$ and $\varphi+1=\varphi^2$, we find \nbea
N_n(k) & \approx & \frac{1}{\sqrt{2 \pi n }} \sqrt{\frac{1- C}{C(1-2C)}} \frac{\sqrt{5}}{\varphi}\cdot\left(1 + O\left(n^{-1/2+\gep}\right)\right)
\\ & = & \frac{1}{\sqrt{2 \pi n }} \sqrt{\frac{(\varphi+1)(\varphi+2)}{\varphi}} \frac{\sqrt{5}}{\varphi} \cdot\left(1 + O\left(n^{-1/2+\gep}\right)\right)
\\ & = & \frac{1}{\sqrt{2 \pi n }} \sqrt{\frac{5(\varphi+2)}{\varphi}} \cdot\left(1 + O\left(n^{-1/2+\gep}\right)\right)
\\ & = & \frac{1}{\sqrt{2 \pi \sigma_{n+1}^2}} \cdot\left(1 + O\left(n^{-1/2+\gep}\right)\right).
\neea

\subsection{Analysis of $S_n(k)$}

\begin{lem}\label{lem:thmmainsnk} For any $\gep > 0$ we have $S_n(k) = \exp(-x^2/2) \exp\left(n^{-1/2+3\gep}\right)$. \end{lem}

\begin{proof}
To understand $$S_n(k) \ = \ \varphi^{-n} \frac{(n-k)^{n-k}}{k^{k} (n-2k)^{n-2k}}$$ we take logarithms and once again change variables by $k = \mu_{n+1} + x \sigma_{n+1}$. We find
\bea\label{eq:bigmessfirst} & &  \log S_n(k) \ = \ \log \left( \varphi^{-n} \frac{(n-k)^{n-k}}{k^{k} (n-2k)^{n-2k}} \right)
\nonumber\\ & = & -n \log\varphi + \left(n-k \right) \log (n-k) -k \log k - \left(n-2k \right) \log (n-2k)
\nonumber\\ & = & -n \log\varphi + \left(n-(\mu_{n+1} + x \sigma_{n+1}) \right) \log (n-(\mu_{n+1} + x \sigma_{n+1}))
\nonumber\\ & & -\left(\mu_{n+1} + x \sigma_{n+1} \right) \log (\mu_{n+1} + x \sigma_{n+1})
\nonumber\\ & & - \left(n-2(\mu_{n+1} + x \sigma_{n+1}) \right) \log (n-2(\mu_{n+1} + x \sigma_{n+1}))
\nonumber\\ & = & -n \log\varphi \nonumber\\ & & + \left(n-(\mu_{n+1} + x \sigma_{n+1}) \right) \left( \log(n-\mu_{n+1}) + \log \left(1 -  \frac{x\sigma_{n+1}}{n-\mu_{n+1}} \right) \right) \nonumber\\ & & -\left(\mu_{n+1} + x \sigma_{n+1} \right) \left( \log(\mu_{n+1}) + \log \left(1 + \frac{x\sigma_{n+1}}{\mu_{n+1}}\right) \right) \nonumber\\ & & - \left(n-2(\mu_{n+1} + x \sigma_{n+1}) \right) \left( \log (n-2\mu_{n+1}) + \log \left(1 - \frac{2x\sigma_{n+1}}{n-2\mu_{n+1}}\right) \right)
\nonumber\\ & = & -n \log\varphi \nonumber\\ & & + \left(n-(\mu_{n+1} + x \sigma_{n+1}) \right) \left(\log \left(\frac{n}{\mu_{n+1}}-1 \right) + \log \left(1 -  \frac{x\sigma_{n+1}}{n-\mu_{n+1}} \right) \right) \nonumber\\ & & -\left(\mu_{n+1} + x \sigma_{n+1} \right) \log \left(1 + \frac{x\sigma_{n+1}}{\mu_{n+1}}\right) \nonumber\\ & & - \left(n-2(\mu_{n+1} + x \sigma_{n+1}) \right) \left(\log \left(\frac{n}{\mu_{n+1}}-2\right) + \log \left(1 - \frac{2x\sigma_{n+1}}{n-2\mu_{n+1}}\right) \right).
\eea

We may simplify the expression above by replacing any $\mu_{n+1}$ inside a logarithm with $Cn$. This is because the resulting Taylor expansion of the logarithms will yield a term of size $1/n^2$. The largest this can be multiplied by is $n$, which leads to an error at most $O(1/n)$. We exponentiate to get $S_n(k)$, which leads to an error factor of size $1-\exp(O(1/n))$, which as $n\to\infty$ is just of size $1/n$. We have $$\log\left(\frac{n}{\mu_{n+1}}-1\right) \ = \ \log\frac{n-\mu_{n+1}}{\mu_{n+1}} \ = \ \log \frac{1-C}{C} + O\left(\frac1{n^2}\right) \ = \ 2\log \phi + O\left(\frac1{n^2}\right),$$ as $(1-C)/C = \phi+1 = \phi^2$. Similarly $$\log\left(\frac{n}{\mu_{n+1}}-2\right) \ = \ \log\frac{n-2\mu_{n+1}}{\mu_{n+1}} \ = \ \log \frac{1-2C}{C} + O\left(\frac1{n^2}\right) \ = \ \log \phi + O\left(\frac1{n^2}\right).$$  Substituting these into \eqref{eq:bigmessfirst} yields \bea & & \log S_n(k) + O\left(\frac1{n}\right) \nonumber\\ & = & -n \log\varphi + \left(n-(\mu_{n+1} + x \sigma_{n+1}) \right) \left(2\log \varphi + \log \left(1 -  \frac{x\sigma_{n+1}}{(1-C)n} \right) \right) \nonumber\\ & & -\left(\mu_{n+1} + x \sigma_{n+1} \right) \log \left(1 + \frac{x\sigma_{n+1}}{Cn}\right) \nonumber\\ & & - \left(n-2(\mu_{n+1} + x \sigma_{n+1}) \right) \left(\log \varphi + \log \left(1 - \frac{2x\sigma_{n+1}}{(1-2C)n}\right) \right).
\eea We note that the coefficient of the $\log\varphi$ term is zero, so \bea  \log S_n(k) + O\left(\frac1{n}\right) & = &\left(n-(\mu_{n+1} + x \sigma_{n+1}) \right)  \log \left(1 -  \frac{x\sigma_{n+1}}{(1-C)n} \right) \nonumber\\ & & -\left(\mu_{n+1} + x \sigma_{n+1} \right) \log \left(1 + \frac{x\sigma_{n+1}}{Cn}\right) \nonumber\\ & & - \left(n-2(\mu_{n+1} + x \sigma_{n+1}) \right)\log \left(1 - \frac{2x\sigma_{n+1}}{(1-2C)n}\right) \nonumber\\  & = &\left((1-C)n - x \sigma_{n+1}) \right)  \log \left(1 -  \frac{x\sigma_{n+1}}{(1-C)n} \right) \nonumber\\ & & -\left(Cn + x \sigma_{n+1} \right) \log \left(1 + \frac{x\sigma_{n+1}}{Cn}\right) \nonumber\\ & & - \left((1-2C)n - 2x \sigma_{n+1}) \right)\log \left(1 - \frac{2x\sigma_{n+1}}{(1-2C)n}\right).
\eea

Let $u = x\sigma_{n+1}/n$. Note $u = O(n^{-1/2+\gep})$, and thus when we expand the logarithms above, we never need to keep more than the $u^2$ terms, as anything further will be small, even upon multiplication by $n$. Expanding gives \bea & &  \log S_n(k) + O\left(n^{-1/2+3\gep}\right) \nonumber\\ &=& (1-C)n \left(-\frac{u}{1-C} - \frac{u^2}{2(1-C)^2}\right) - un \left(-\frac{u}{1-C}\right) \nonumber\\ & & -Cn \left(\frac{u}{C} - \frac{u^2}{2C^2}\right) - un \left(\frac{u}{C}\right) \nonumber\\ & & -(1-2C)n \left(-\frac{2u}{1-2C}-\frac{4u^2}{2(1-2C)^2}\right) + 2un \left(-\frac{2u}{1-2C}\right). \eea Note all the $un$ terms cancel, and all that survives are the $u^2n$ terms. Thus \bea & &  \log S_n(k) + O\left(n^{-1/2+3\gep}\right) \nonumber\\ &=& \left[-\frac1{2(1-C)} + \frac1{1-C} + \frac1{2C} - \frac1{C} + \frac2{1-2C} - \frac4{1-2C}\right] u^2n \nonumber\\ &=&  -\frac{u^2 n}{2 C - 6 C^2 + 4 C^3}.  \eea As $C = \frac1{\varphi+2}$ and $u = x\sigma_{n+1}/n$ with $\sigma_{n+1}^2 = \frac{\varphi n}{5(\varphi+2)} + O(1)$, simplifying the above yields \be \log S_n(k) + O\left(n^{-1/2+3\gep}\right) \ = \ -\frac{x^2}{2}, \ee and thus exponentiating gives \be S_n(k) \ = \ \exp(-x^2/2) \exp\left(n^{-1/2+3\gep}\right).  \ee \end{proof}

\subsection{Proof of Theorem \ref{thm:main}}

Using the results from the previous subsections (Lemmas \ref{lem:thmmainnnk} and \ref{lem:thmmainsnk}, and the change of variable argument on why a factor of $\sigma_{n+1}$ enters), we can now prove the convergence to a Gaussian.

\begin{proof}[Proof of Theorem \ref{thm:main}] The only item left to prove in Theorem \ref{thm:main} is the convergence to the Gaussian. We have \bea p_{n+1}(k)dk  & \ = \ & p_{n+1}\left(\mu_{n+1} + \sigma_{n+1}x\right) \sigma_{n+1} dx \nonumber\\ &=& N_n(k) S_n(k) \sigma \left(1 + O(n^{-1/2+3\gep})\right) dx \nonumber\\ &=&  \frac{1}{\sqrt{2 \pi \sigma_{n+1}^2}} \cdot e^{-x^2/2} \cdot \sigma_{n+1} \left(1 + O(n^{-1/2+3\gep})\right)dx \nonumber\\ &=& \frac1{\sqrt{2\pi}}\ e^{-x^2/2} \left(1 + O(n^{-1/2+3\gep})\right)dx,\eea which clearly converges to the standard normal as $n\to\infty$.\end{proof}


\section{Far-difference representations}

We now consider the problem of the far-difference representations. We are studying the number of positive and negative Fibonacci summands in the decomposition of integers in $[S_n, S_{n+1})$, where $S_n = \sum_{0<n-4i\le n} F_{n-4i}$. Using combinatorial arguments as in Lemma \ref{lem:cookieproblem} (which involve solving a variant of this Diophantine problem inside a variant of this Diophantine problem),
we can come up with a formula for the joint density function $p_n(k,\ell)$ for the number of integers in $[S_n, S_{n+1})$ with exactly $k$ positive Fibonacci summands and exactly $\ell$ negative Fibonacci summands. After a lot of algebra, we obtain the formula:
\begin{eqnarray}
p_n(k,\ell) & \ = \ & \sum_{r=0}^k \ncr{k-1}{k-r} \Bigg[ \ncr{\ell-1}{r-2}\ncr{n-3(k+\ell)+2r-3}{k+\ell}\nonumber\\
& & \ \ + \ \ncr{\ell-1}{r-1}\ncr{n-3(k+\ell)+2r-2}{k+\ell}+
\ncr{\ell-1}{r-1}\ncr{n-3(k+\ell)+2r-1}{k+\ell}\nonumber\\
& & \ \ + \ \ncr{\ell-1}{r}\ncr{n-3(k+\ell)+2r}{k+\ell} \Bigg].
\end{eqnarray}

If we could get good asymptotics for this formula, we could calculate the limiting density directly, but this does not seem feasible.  With enough patience, this formula can be used to calculate any particular joint moment, but this becomes cumbersome very general. Unlike our expression for the standard Fibonacci case, here we have sums and products of binomial coefficients that would need to be simplified before we can fruitfully apply Stirling's formula. As the generating function technique is able to handle this problem, we invite the reader to see \cite{MW} for a complete analysis of this problem.


\section{Conclusion and Future Research}

Our combinatorial viewpoint has allowed us to extend previous work and obtain Gaussian behavior for the number of summands for a large class of recurrence relations. This is just the first of many questions one can ask. Others, which we hope to return to at a later date, include:

\begin{enumerate}

\item Lekkerkerker's theorem, and the Gaussian extension, are for the behavior in intervals $[F_n, F_{n+1})$.
Do the limits exist if we considere other intervals, say
$[F_n+g_1(F_n), F_n + g_2(F_n))$ for some functions $g_1$ and $g_2$? If yes, what must be
true about the growth rates of $g_1$ and $g_2$?

\item For the generalized recurrence relations, what happens if instead of looking at $\sum_{i=1}^n a_i$ we study $\sum_{i=1}^n \min(1,a_i)$? In other words, we only care about how many distinct $H_i$'s occur in the decomposition.

\item What can we say about the distribution of the largest gap between summands in the Zeckendorf decomposition? Appropriately normalized, how does the distribution of gaps between the summands behave?

\end{enumerate}


\appendix

\section{Proofs of Combinatorial Identities}\label{sec:appcombid}

We collect the proofs of the various needed combinatorial identities.

\begin{lem}\label{lem:sumupdownfib} Let $F_m$ denote the
$m$\textsuperscript{th} Fibonacci number, with $F_1 = 1$, $F_2 = 2$,
$F_3=3$, $F_4 = 5$ and so on. Then \be \sum_{k=0}^{\lfloor
\frac{n-1}2\rfloor} \ncr{n-1-k}{k} \ = \ F_{n-1}. \ee \end{lem}

\begin{proof} We proceed by induction. The base case is trivially
verified, and by brute force checking we may assume $n \ge 4$. We
assume our claim holds for $n$ and must show that it holds for $n+1$.
Note that we may extend the sum to $n-1$, as $\ncr{n-1-k}{k} = 0$
whenever $k > \lfloor \frac{n-1}2\rfloor$. Using the standard
identity that \be \ncr{m}{\ell} + \ncr{m}{\ell+1} \ = \
\ncr{m+1}{\ell+1},\ee and the convention that $\ncr{m}{\ell} = 0$ if
$\ell$ is a negative integer, we find \bea \sum_{k=0}^{n}
\ncr{n-k}{k} & \ = \ & \sum_{k=0}^n \left[\ncr{n-1-k}{k-1} +
\ncr{n-1-k}{k}\right] \nonumber\\ &=& \sum_{k=1}^n \ncr{n-1-k}{k-1} +
\sum_{k=0}^n \ncr{n-1-k}{k} \nonumber\\ &=& \sum_{k=1}^n
\ncr{n-2-(k-1)}{k-1} + \sum_{k=0}^n \ncr{n-1-k}{k} \nonumber\\ &=&
F_{n-2} + F_{n-1} \eea by the inductive assumption; noting
$F_{n-2}+F_{n-1} = F_n$ completes the proof. \end{proof}

\begin{proof}[Proof of Lemma \ref{lem:recurrencecen}]\label{lem:proofrec} The lemma follows from straightforward algebra. We have \bea
\ce(n) & \ = \ & \sum_{k=0}^{\lfloor \frac{n-1}2\rfloor} k
\ncr{n-1-k}{k} \nonumber\\ &=& \sum_{k=1}^{\lfloor
\frac{n-1}2\rfloor} k \frac{(n-1-k)!}{k!(n-1-2k)!}
\nonumber\\ &=& \sum_{k=1}^{\lfloor \frac{n-1}2\rfloor} (n-1-k)
\frac{(n-2-k)!}{(k-1)!(n-1-2k)!} \nonumber\\ &=& \sum_{k=1}^{\lfloor
\frac{n-1}2\rfloor} (n-2-(k-1))
\frac{(n-3-(k-1)!}{(k-1)!(n-3-2(k-1))!} \nonumber\\ &=&
\sum_{\ell=0}^{\lfloor \frac{n-3}2\rfloor} (n-2-\ell)
\ncr{n-3-\ell}{\ell} \nonumber\\ &=& (n-2) \sum_{\ell=0}^{\lfloor
\frac{n-3}2\rfloor} \ncr{n-3-\ell}{\ell} - \sum_{\ell=0}^{\lfloor
\frac{n-3}2\rfloor} \ell \ncr{n-3-\ell}{\ell} \nonumber\\ &=& (n-2)
F_{n-3} - \ce(n-2), \eea which establishes the lemma. Note that we
used the binomial identity again (Lemma \ref{lem:sumupdownfib}) to
replace the sum of binomial coefficients with a Fibonacci number.
\end{proof}

\begin{proof}[Proof of Lemma \ref{lem:formula}]\label{lem:proofformula} Consider \bea & &
\sum_{\ell = 0}^{\lfloor \frac{n-3}{2}\rfloor} (-1)^\ell
\left(\ce(n-2\ell) + \ce(n-2(\ell+1))\right) \ = \
\sum_{\ell=0}^{\lfloor \frac{n-3}{2}\rfloor} (-1)^\ell (n-2-2\ell)
F_{n-3-2\ell} \nonumber\\ & & \ \ \ \ \ \ \ \ \ \ \ \ \ \ \ \ = \
\sum_{\ell=0}^{\lfloor \frac{n-3}{2}\rfloor} (-1)^\ell (n-3-2\ell)
F_{n-3-2\ell} + \sum_{\ell=0}^{\lfloor \frac{n-3}{2}\rfloor}
(-1)^\ell (2\ell) F_{n-3-2\ell} \nonumber\\  & & \ \ \ \ \ \ \ \ \ \
\ \ \ \ \ \ = \ \sum_{\ell=0}^{\lfloor \frac{n-3}{2}\rfloor}
(-1)^\ell (n-3-2\ell) F_{n-3-2\ell} + O(F_{n-2}); \eea while we could
evaluate the last sum exactly, trivially estimating it suffices to
obtain the main term (as we have a sum of every other Fibonacci
number, the sum is at most the next Fibonacci number after the
largest one in our sum).

We now use Binet's formula (see \eqref{eq:binetformula}) to convert the sum into a geometric
series. Letting $\varphi = \frac{1+\sqrt{5}}2$ be the golden mean, we
have \be F_n \ = \ \frac{\varphi}{\sqrt{5}} \cdot \varphi^n -
\frac{1-\varphi}{\sqrt{5}} \cdot (1-\varphi)^n \ee (our constants are
because our counting has $F_1 = 1$, $F_2 =2$ and so on). As
$|1-\varphi| < 1$, the error from dropping the $(1-\varphi)^n$ term
is $O(\sum_{\ell\le n} n) = O(n^2) = o(F_{n-2})$, and may thus safely
be absorbed in our error term. We thus find \bea \ce(n) & \ = \ &
\frac{\varphi}{\sqrt{5}} \sum_{\ell=0}^{\lfloor \frac{n-3}{2}\rfloor}
(n-3-2\ell) (-1)^\ell \varphi^{n-3-2\ell} + O(F_{n-2}) \nonumber\\
&=& \frac{\varphi^{n-2}}{\sqrt{5}} \left[(n-3) \sum_{\ell=0}^{\lfloor
\frac{n-3}{2}\rfloor} (-\varphi^{-2})^{\ell} -2
\sum_{\ell=0}^{\lfloor \frac{n-3}{2}\rfloor} \ell
(-\varphi^{-2})^{\ell} \right] + O(F_{n-2}).\eea

We use the geometric series formula to evaluate the first term. We
drop the upper boundary term of $(-\varphi^{-1})^{\lfloor
\frac{n-3}2\rfloor}$, as this term is negligible since $\varphi > 1$.
We may also move the 3 from the $n-3$ into the error term, and are
left with \bea \ce(n) & \ = \ & \frac{\varphi^{n-2}}{\sqrt{5}}
\left[\frac{n}{1+\varphi^{-2}} -2 \sum_{\ell=0}^{\lfloor
\frac{n-3}{2}\rfloor} \ell (-\varphi^{-2})^{\ell} \right] +
O(F_{n-2}) \nonumber\\ & \ = \ & \frac{\varphi^{n-2}}{\sqrt{5}}
\left[\frac{n}{1+\varphi^{-2}} -2
S\left(\Big\lfloor\frac{n-3}{2}\Big\rfloor, -\varphi^{-2}\right)
\right] + O(F_{n-2}), \eea
where \be \mathcal{S}(m,x) \ = \ \sum_{j=0}^m j x^j. \ee There is a
simple formula for $\mathcal{S}(m,x)$. As \be \sum_{j=0}^m x^j \ = \
\frac{x^{m+1}-1}{x-1},  \ee  applying the operator $x \frac{d}{dx}$
gives \be \mathcal{S}(m,x) \ = \ \sum_{j=0}^m j x^j \ = \ x
\frac{(m+1)x^m(x-1)-(x^{m+1}-1)}{(x-1)^2} \ = \
\frac{mx^{m+2}-(m+1)x^{m+1}+x}{(x-1)^2}. \ee Taking $x =
-\varphi^{-2}$, we see that the contribution from this piece may
safely be absorbed into the error term $O(F_{n-2})$, leaving us with
\be \ce(n) \ = \ \frac{n \varphi^{n-2}}{\sqrt{5}(1+\varphi^{-2})} +
O(F_{n-2})\ = \ \frac{n \varphi^n}{\sqrt{5}(\varphi^2+1)} +
O(F_{n-2}). \ee Noting that for large $n$ we have $F_{n-1} =
\frac{\varphi^n}{\sqrt{5}}  + O(1)$, we finally obtain \be \ce(n) \ =
\ \frac{n F_{n-1}}{\varphi^2+1} + O(F_{n-2}).\ee A more careful analysis is possible; such a computation leads to the exact form for the mean given in \eqref{eq:meann}.
\end{proof}


\ \\

\end{document}